%% file: lipschitz-push-forward.tex
\documentclass{article}

\input{preamble}

\title{Lipschitz changes of variables via heat flow}
\author{Joe Neeman}

\begin{document}

\maketitle

\begin{abstract}
We extend Caffarelli's contraction theorem, by proving that
there exist a Lipschitz changes of variables between the Gaussian measure
and certain perturbations of it. Our approach is based on an argument
due to Kim and Milman, in which the changes of variables are constructed
using a heat flow.
\end{abstract}

\section{Introduction}
Caffarelli~\cite{Caffarelli:00} showed that every probability measure that is
log-concave with respect to Gaussian can be realized as a contraction of the
Gaussian measure: if $d\mu = e^{-V(x)} \, d\gamma$ is a probability measure
(here and throughout, $\gamma$ denotes the standard Gaussian measure on $\R^n$),
and $V$ is
convex, then there is contraction $T: \R^n \to \R^n$ such that $T_\# \gamma =
\mu$. (Here $T_\#$ denotes the push-forward of a measure under $T$, defined by
$(T_\# \gamma)(A) = \gamma(T^{-1}(A))$ for all Borel sets $A$.) Caffarelli's
contraction theorem has found many applications (see~\cite{CoFiJh:17} for more background), mainly
because many important geometric and functional
inequalities are preserved under contractions.
Hence, if one is able such an inequality for $\gamma$, it also holds for every
push-forward of $\gamma$ under a contraction.

It is natural to ask whether contractions (or, more generally, Lipschitz maps)
pushing $\gamma$ onto $\mu = e^{-V(x)}\, d\gamma$ can be found under more
general conditions. In the one-dimensional case, Bobkov~\cite{Bobkov:10}
showed that a bounded perturbation of a log-concave probability measure
is a Lipschitz image of the Gaussian measure:

\begin{theorem}\label{thm:bobkov}
    Suppose that $V^*: \R \to \R$ is a convex function, and suppose that $e^{-V} d\gamma$
    is a probability measure on $\R$, where $V^*(x) \le V(x) \le V^*(x) + c$ for all $x \in \R$.
    Then there is an $e^c$-Lipschitz map pushing $\gamma$ onto $e^{-V} d\gamma$.
\end{theorem}

By combining this result with a localization argument, Bobkov was able to deduce functional
inequalities for perturbations of Gaussian measure in $\R^n$. However, he did not
establish the existence of Lipschitz changes of variables between $n$-dimensional
probability measures.

Colombo et al.~\cite{CoFiJh:17} studied Lipschitz changes of variables between more
general measures, but for the sake of brevity we will quote their result in the case
that the underlying measure is Gaussian.

\begin{theorem}\label{thm:colombo}
    Suppose that $e^{-V} d\gamma$ is a probability measure on $\R^n$,
    where $D^2 V \ge -\lambda$ (for $\lambda \ge 1$), and $V$ is constant outside of $B(0, R)$
    (the Euclidean ball of radius $R$). Then $e^{-V} d\gamma$ is an $L$-Lipschitz push-forward
    of $\gamma$, for some $L$ depending on $R$ and $\lambda$.
\end{theorem}

Our main result is to replace the compactly supported perturbation of Theorem~\ref{thm:colombo}
with a bounded perturbation.

\begin{theorem}\label{thm:main}
    Suppose that $d\mu = e^{-V(x)} d\gamma$ is a probability measure on $\R^n$,
    where $D^2 V \ge -\lambda$ (for $\lambda \ge 1$), and $\sup V - \inf V \le c$.
    Then $\mu$ is an $L$-Lipschitz push-forward of $\gamma$, for
    \[
        L = 2 (2\lambda)^{e^{c}}.
    \]
\end{theorem}

We remark that our techniques are related to those of Mikulincer
and Shenfeld~\cite{MikulincerShenfeld:2022}, who showed the existence of
Lipschitz push-forwards in several other situations.

\section{The heat-flow map}

Whereas Caffarelli~\cite{Caffarelli:00} and Colombo et al.~\cite{CoFiJh:17}
studied the Brenier map arising from optimal transport, we will establish
Theorem~\ref{thm:main} using a map coming from the heat flow. This map
previously appeared~\cite{OttoVillani:00} in the study of functional inequalities.
It was used by Kim and Milman~\cite{KimMilman:12} to give an alternate
proof of Caffarelli's contraction theorem, and it was further studied by Tanana~\cite{Tanana:17},
who showed that in general it differs from the Brenier map.

We begin by defining the Ornstein-Uhlenbeck semigroup: for $t > 0$ and $f: \R^n \to \R$
bounded and measurable, define $P_t f$ by
\[
    (P_t f)(x) = \E f(e^{-t} x + \sqrt{1-e^{-2t}} Z),
\]
where $Z$ denotes a standard Gaussian random variable. One can easily see that $P_t f$
is also a bounded, measurable function, and that $P_t$ is a semigroup in
the sense that $P_t P_s f = P_{t+s} f$ for any $s$, $t$, and $f$. Moreover,
Jensen's inequality implies that $P_t$ is a contraction on every $L^p(\gamma)$ space,
meaning that we can extend $P_t$ to $L^p(\gamma)$ by density.
It is well-known that for every $1 \le p < \infty$, $P_t f$ converges
in $L^p(\gamma)$ to the constant $\E f$ as $t \to \infty$.

From now on, we will adopt the convention that $f = e^{-V}$ is a probability density with
respect to $\gamma$. We write $f_t$ for $P_t f$ and $V_t$ for $-\log P_t f$.
The main idea of Kim and Milman~\cite{KimMilman:12} is to consider a flow $S_t: \R^n \to \R^n$
satisfying
\begin{equation}\label{eq:St}
    \frac{dS_t(x)}{dt} = (\nabla V_t)(S_t(x)).
\end{equation}
Assuming for the moment that such a flow exists, it is not hard to show that it satisfies
$(S_t)_\# (fd\gamma) = f_t d\gamma$. In the $t \to \infty$ limit, $f_t \to 1$ and
we obtain a map pushing $f d\gamma$ onto $\gamma$.

In order to have sufficient regularity for the existence of $S_t$, we will
apply an approximation argument: we will first establish the result for
sufficiently regular $f$, and then take a limit. This procedure is
straightforward, and has already been discussed in~\cite{KimMilman:12}.
We will include a sketch, however.

We first observe that in order to establish the existence of Lipschitz
maps between two measures, it suffices to approximate both of them in distribution:

\begin{lemma}\label{lem:approximation}
    Let $\mu$ and $\nu$ be two probability measures on $\R^n$; suppose that $\nu_k$
    and $\mu_k$ are sequences of probability measures
    converging in distribution to $\nu$ and $\mu$ respectively.
    Suppose, moreover, that for every $k$ there exists an $L$-Lipschitz map $T_k$
    with $(T_k)_\# \nu_k = \mu_k$. Then there is a $L$-Lipschitz map $T$
    with $T_\# \nu = \mu$.
\end{lemma}

\begin{proof}
    Let $\Pi$ denote the set of probability measures on $\R^n \times \R^n$,
    and let $\Pi(\nu,\mu)$ denote the set of couplings between $\nu$ and $\mu$; i.e., those
    elements of $\Pi$ whose marginals are $\nu$ and $\mu$
    respectively. We say that $\pi \in \Pi$ is $L$-Lipschitz if for every
    $(w, x), (y, z) \in \supp(\pi)$, $|x - z| \le L |w - y|$.
    Clearly, if $T$ is an $L$-Lipschitz map with $T_\# \nu = \mu$ then $(T \otimes \Id)_\# \nu$
    is an $L$-Lipschitz coupling between $\nu$ and $\mu$. On the other hand, if $\pi \in \Pi(\nu,\mu)$
    is an $L$-Lipschitz coupling, it follows that for every $x \in \supp \nu$ there is a unique
    $y \in \supp \mu$ such that $(x, y) \in \supp \pi$. The map $T: \supp \mu \to \R^n$ defined by
    $x \mapsto y$ is $L$-Lipschitz (and hence measurable). By MacShane's lemma,
    it can be extended to an $L$-Lipschitz map on all of $\R^n$, and the resulting map satisfies
    $T_\# \nu = \mu$. In other words, $L$-Lipschitz couplings are in one-to-one correspondence
    with $L$-Lipschitz maps.

    Note that the set of $L$-Lipschitz couplings is closed with respect to convergence
    in distribution: if $\pi_k \toD \pi$ then every $(w, x) \in \supp \pi$ can be expressed
    as the limit of $(w_k, x_k) \in \supp \pi_k$, and it follows that if $\pi_k$ is $L$-Lipschitz
    then so is $\pi$. Now, the assumption of the lemma implies that there exists a $L$-Lipschitz
    $\pi_k \in \Pi(\nu_k, \mu_k)$ for every $k$. Since $\Pi$ is compact with respect to convergence
    in distribution, after passing to a subsequence we may find $\pi \in \Pi$
    such that $\pi_k \toD \pi$. It follows easily that $\pi \in \Pi(\nu, \mu)$, and the earlier
    argument implies that $\pi$ is $L$-Lipschitz. By the previous paragraph, there
    is an $L$-Lipschitz map pushing $\nu$ onto $\mu$.
\end{proof}

We claim that whenever $V(x)$ is $C^\infty$ and $\nabla V(x)$ is uniformly bounded,
there is a unique solution $S_t$ solving~\eqref{eq:St} for all $t \ge 0$.
Indeed, the classical Picard-Lindel\"of theory implies the existence and uniqueness
of such an $S_t$ provided that $\nabla V_t$ is uniformly bounded
and locally spatially Lipschitz. Now, the definition of $P_t$ immediately implies
that if $V$ (and hence $f$) is $C^\infty$ then $P_t f$ is $C^\infty$ in both $x$ and $t$,
and strictly positive for all $t$. Hence, $\nabla V_t$ is $C^\infty$ in both $x$
and $t$, implying in particular that it is locally spatially Lipschitz. To
prove the boundedness of $\nabla V_t = \frac{\nabla P_t f}{P_t f}$, note that
\[
    \nabla P_t f(x) = e^{-t} P_t \nabla f(x) = -e^{-t} P_t (f \nabla V)(x).
\]
It follows that $|\nabla V_t(x)| \le e^{-t} \|\nabla V\|_\infty$ for every $x$ and $t$.

The next task is to show that $f d\gamma$ can be approximated in distribution by
probability measures satisfying our regularity assumptions. That is, we will show
that for any measurable $V$ satisfying $\int e^{-V} \, d\gamma = 1$, there exists
a sequence $V_k \in C^\infty$ such that $\int e^{-V_k} \, d\gamma = 1$
for every $k$, $\|\nabla V_k\|_\infty < \infty$ for every $k$,
and $e^{-V_k} d\gamma \toD e^{-V} d\gamma$. We can construct an approximation to $e^{-V}d\gamma$
in three steps: first, any measurable $V$ can be approximated by $C^\infty$ functions (for example,
by convolving $f = e^{-V}$ with a Gaussian density). Next, any $C^\infty$ $V$
can be approximated by Lipschitz functions (for example, by replacing $V$ with
$\tilde V(x) := \inf\{V(x) + L |x - y| : |y| \le R\} - c$ for $L$ and $R$
sufficiently large, where $c$ is a normalizing constant to ensure that
$e^{-\tilde V}$ is a probability density). Finally, any Lipschitz $V$ can be approximated
by $C^\infty$, Lipschitz functions (for example, by convolving with a Gaussian again).

Together with Lemma~\ref{lem:approximation}, this approximation argument allows
us to assume, from now on, that $V$ is $C^\infty$ and Lipschitz, and that $S_t$
exists.

\section{Lipschitz maps from log-concavity}

Let us briefly recall the argument of Kim and Milman. If we assume that $V$ is convex,
it follows from the Prekopa-Leindler inequality that $V_t$ is concave.
Since $\inr{x - y}{\nabla g(x) - \nabla g(y)} \ge 0$ for any convex function $g$
and any $x, y \in \R^n$, we have
\[
    \diff{}{t} (|S_t(x) - S_t(y)|^2) = 2 \inr{S_t(x) - S_t(y)}{\nabla V_t(S_t(x)) - \nabla V_t(S_t(y))} \ge 0.
\]
In particular, $S_t$ satisfies $|S_t(x) - S_t(y)| \ge |x - y|$ for all $x, y, t$.
Then $S_t^{-1}$ is a contraction pushing $f_t \gamma$ onto $f d\gamma$. Taking
$t \to \infty$ and applying Lemma~\ref{lem:approximation}, there is a contraction
pushing $\gamma$ onto $f d\gamma$.

The main observation of this work is that we may still obtain something under a weakening
of log-concavity.

\begin{definition}
    Say that $g: \R^n \to \R$ is $-\lambda$-concave if $g(x) - \lambda |x|^2/2$ is concave,
	and say that $g: \R^n \to \R$ is $-\lambda$-convex if $-g$ is $-\lambda$-concave.
    Say that $g: \R^n \to (0, \infty)$ is $-\lambda$-log-concave if $\log g$ is $-\lambda$-concave.
\end{definition}

We will be mainly interested in the case $\lambda \ge 1$, since the case
$\lambda < 1$ is already covered by Caffarelli's theorem: if $f$ is
$-\lambda$-log-concave for $\lambda < 1$ then $\tilde f(x) = c
f(\sqrt{1-\lambda} x)\exp(\lambda(1-\lambda) |x|^2/2)$ (for a normalizing
constant $c$) is log-concave and satisfies $f(x) e^{-|x|^2/2} = \tilde f(y)
e^{-|y|^2/2}$ where $y = x/\sqrt{1-\lambda}$.  Hence, the dilation by a factor
of $\frac{1}{\sqrt{1-\lambda}}$ pushes $\tilde f d\gamma$ onto $f d\gamma$.  On
the other hand, Caffarelli's theorem shows that there is a contraction pushing
$\gamma$ onto $\tilde f d\gamma$, and composing these two maps gives a
$\frac{1}{\sqrt{1-\lambda}}$-Lipschitz function pushing $\gamma$ onto $f
d\gamma$. Hence, we focus on $-\lambda$-concave functions for $\lambda \ge 1$.

A minor extension of Kim and Milman's argument yields the following:

\begin{lemma}\label{lem:KM}
    If there is some integrable $\lambda: [0, \infty) \to \R$ such that $f_t$ is
    $-\lambda(t)$-log-concave for every $t \ge 0$ then $f \, d\gamma$ is an $L$-Lipschitz
    image of $\gamma$ for
    \[
        L = \exp\left(\int_0^\infty \lambda(s)\, ds\right).
    \]
\end{lemma}

The Kim-Milman contraction proof is the special case of the above where $\lambda
\equiv 0$, but we obtain a non-trivial result whenever $\lambda$ is integrable.

\begin{proof}
    A $-\lambda$-convex function $g$ satisfies $\inr{x - y}{\nabla g(x) - \nabla g(y)} \ge -\lambda |x - y|$
    for every $x$ and $y$; the assumption of the lemma implies that $V_t$ is $-\lambda(t)$-convex for every $t$.
    Hence, if $S_t$ is the flow defined above then for any $x, y$,
    \begin{align*}
        \diff{}{t} (|S_t(x) - S_t(y)|^2)
        &= 2 \inr{S_t(x) - S_t(y)}{\nabla V_t(S_t(x)) - \nabla V_t(S_t(y))} \\
        &\ge -2 \lambda(t) |S_t(x) - S_t(y)|^2.
    \end{align*}
    It follows that the function $h(t) = |S_t(x) - S_t(y)|^2$ satisfies $(\log h)'(t) \ge -2 \lambda(t)$ and hence
    \[
        \frac{h(t)}{h(0)} \ge \exp\left(-\int_0^t 2 \lambda(s)\, ds\right).
    \]
    It follows that 
    \[
        |S_t(x) - S_t(y)| \ge |x - y| \exp\left(-\int_0^t \lambda(s)\, ds\right),
    \]
    meaning that $S_t^{-1}$ is an $\exp\left(\int_0^t \lambda(s)\, ds\right)$-Lipschitz map
    pushing $f_t d\gamma$ onto $f d\gamma$. Taking $t \to \infty$ and applying
    Lemma~\ref{lem:approximation} completes the proof.
\end{proof}

Our remaining task is to find conditions on $f$ (or, equivalently, $V$) to ensure that $f_t$
is $-\lambda(t)$-log-concave.
We can obtain an easy bound if we assume an upper bound on $D^2 f$ and a lower bound on $f$:

\begin{proposition}
    For any $f$ bounded away from zero satisfying $D^2 f \le C \Id$, $f_t$ is $-\lambda(t)$-log-concave for
    \[
        \lambda(t) = \frac{C e^{-2t}}{\inf f}.
    \]
    In particular, under these assumptions there is a $\frac{C}{2 (\inf f)^2}$-Lipschitz map
    pushing $\gamma$ onto $f d\gamma$.
\end{proposition}

\begin{proof}
    We can write $D^2 P_t f = e^{-2t} P_t (D^2 f) \le C e^{-2t} I$.
    Hence,
    \[
        D^2 V_t = \frac{D^2 f_t}{f_t} - \frac{(\nabla f_t)^{\otimes 2}}{f_t^2} \le \frac{C e^{-2t} \Id}{\inf f}
    \]
    The second claim follows by applying Lemma~\ref{lem:KM} to the first.
\end{proof}

It is natural to attempt to replace the assumption on $D^2 f$ with an
assumption on $D^2 V$ (for example, because $D^2 V$ behaves more
predictably under tensorization). As a first step, we observe that a bound
on $D^2 V$ implies a bound on $D^2 V_t$ for $t$ sufficiently small:

\begin{lemma}\label{lem:log-concave-implies-log-concave}
    Suppose $f$ is $-\lambda$-log-concave (for $\lambda \ge 0$). For every $t > 0$ satisfying $\lambda(1-e^{-2t}) < 1$,
    $f_t$ is $-\frac{\lambda e^{-2t}}{1-\lambda(1-e^{-2t})}$-log-concave.
\end{lemma}

As we will see, the range of $t$ in
Lemma~\ref{lem:log-concave-implies-log-concave} cannot be improved.

The proof of Lemma~\ref{lem:log-concave-implies-log-concave} is a
straightforward modification of a standard argument using Prekop\'a's theorem: every marginal
of a log-concave function is log-concave.

\begin{proof}[Proof of Lemma~\ref{lem:log-concave-implies-log-concave}]
	Let $\tilde f(x) = f(x) \exp(-\lambda x^2/2)$, so that $\tilde f$ is
	log-concave. Since composition with an affine function preserved log-concavity,
	$\tilde f(e^{-x} t + \sqrt{1-e^{-2t}} y)$ is log-concave in both $x$ and $y$.
	Define
	\begin{align*}
        \alpha &= 1 - \lambda(1-e^{-2t}) \\
        \beta &= \frac{\lambda e^{-t} \sqrt{1-e^{-2t}}}{\alpha};
	\end{align*}
    then $\alpha > 0$ and so
    \begin{multline*}
        \tilde f(e^{-t} x + \sqrt{1-e^{-2t}} y) \exp\left(-\frac{\alpha}{2} (y - \beta x)^2\right) \\
        = f(e^{-t} + \sqrt{1-e^{-2t}} y) \exp\left(-\frac 12 (\lambda e^{-2t} + \alpha \beta^2) x^2 - \frac 12 y^2\right)
    \end{multline*}
    is log-concave in $x$ and $y$.
    Integrating out $y$ will leave (by Prekop\'a's theorem) a function that is log-concave in $x$.
    On the other hand, the definition of $P_t$ implies that integrating out $y$ yields
    \[
    (2\pi)^{-n/2} P_t f(x) \exp\left(-\frac 12 (\lambda e^{-2t} + \alpha \beta^2) x^2\right),
    \]
    and hence $P_t f$ is $-(\lambda e^{-2t} + \alpha \beta^2)$-log-concave. Plugging
    in the definitions of $\alpha$ and $\beta$ proves the claim.
\end{proof}

To see that the range of $t$ in Lemma~\ref{lem:log-concave-implies-log-concave} is optimal,
consider the function $g(x) = \min\{e^{T^2/2}, e^{x^2/2}\}$ for a parameter $T$.
Then $g(x)$ is $-1$-log-concave, and $x \mapsto g(\lambda x)$ is $-\alpha$-log-concave.
A straightforward computation shows that if $Z$ is a standard Gaussian random variable then
\[
    h(x) := \E g(x + Z) = \Phi(-x-T) + \Phi(x-T) + \phi(x) \frac{e^{xT} - e^{-xT}}{x},
\]
where $\phi(x) = \frac{1}{\sqrt{2\pi}} e^{-x^2/2}$, $\Phi(x) = \int_{-\infty}^x \phi(y)\, dy$, and
$\frac{e^{xT} - e^{-xT}}{x}$ is understood to equal $2T$ when $x=0$. In particular,
$h(0) = 2\Phi(-T) + \frac{\sqrt 2}{\sqrt \pi} T$, $h'(0) = 0$, and
\[
    h''(0)  = 2 T (\phi(T) - \phi(0)) + \frac{\sqrt 2}{3 \sqrt \pi} T^3.
\]
Now fix $t$ and define $f(x) = g(\frac{x}{\sqrt{1-e^{-2t}}})$; thus $f$ is
$-\lambda$-log-concave for $\lambda = \frac{1}{\sqrt{1-e^{-2t}}}$. Then $f_t(x) = \E
g((e^{2t} - 1)^{-1/2} x + Z) = h((e^{2t} - 1)^{1/2} x)$. It follows that
$f_t(0) \le \frac{\sqrt 2}{\sqrt \pi} T$, $f_t'(0) = 0$, and
$f_t''(0) \ge (e^{2t} - 1)^{-1} \frac{\sqrt 2}{3 \sqrt \pi} T^3$. In particular,
\[
    (\log f_t)''(0) = \frac{f_t''(0)}{f_t(0)} \ge \frac{T^2}{3(e^{2t} - 1)};
\]
since $T$ is arbitrary, it is clear that we cannot bound the log-concavity
of $f_t$ when $\lambda (1-e^{-2t}) = 1$.

The fact that Lemma~\ref{lem:log-concave-implies-log-concave} does not apply for every $t$
means that we will require other assumptions on $f$ in order to use Lemma~\ref{lem:KM}.
Boundedness of $V$ turns out to be sufficient:

\begin{lemma}\label{lem:bounded-implies-log-concave}
    If $f = e^{-V}$ where $\sup V - \inf V \le c$ then $f_t$ is $-\frac{e^{c}}{e^{2t} - 1}$-log-concave.
\end{lemma}

\begin{proof}
    For a unit vector $v \in \R^n$,
    \[
        D^2_{v,v} f_t(x) = \frac{\E[(\inr{X}{v}^2 - 1) f(e^{-x} + \sqrt{1-e^{-2t}} X)]}{e^{2t} - 1}
        \le \frac{\|f\|_\infty \E[\inr{X}{v}^2]}{e^{2t} - 1} = \frac{e^{-\inf V}}{e^{2t} - 1}.
    \]
    On the other hand, $f_t(x) \ge e^{-\sup V}$ for every $x$, and so
    \[
        D^2_{v,v} \log f_t(x) \le \frac{D^2_{v,v} f_t(x)}{f_t(x)} \le \frac{e^{\sup V - \inf V}}{e^{2t} - 1}.
    \]
\end{proof}

The proof of Theorem~\ref{thm:main} now follows easily: if $f$ satisfies the assumptions of Theorem~\ref{thm:main}
then $f_t$ is $-\lambda(t)$-log-concave for
\[
    \lambda(t) = \min\left\{
        \frac{\lambda e^{-2t}}{1-\lambda(1-e^{-2t})}, \frac{e^c}{e^{2t} - 1}
    \right\}.
\]
Note that
\[
    \int_s^\infty \frac{1}{e^{2t} - 1}\, dt = -\frac 12 \log(1-e^{-2s})
\]
and
\[
    \int_0^s \frac{\lambda e^{-2t}}{1 - \lambda(1-e^{-2t})}\, dt = -\frac 12 \log (1-\lambda(1-e^{-2s})).
\]
If we choose $s$ so that $1 - e^{-2s} = \frac{1}{2\lambda}$, we get the bound
\begin{align*}
    \int_0^\infty \lambda(t)\, dt
    &\le \int_0^s \frac{\lambda e^{-2t}}{1 - \lambda(1-e^{-2t})}\, dt
    + \int_s^\infty \frac{e^c}{e^{2t} - 1}\, dt \\
    &= \frac{e^c}{2} \log(2\lambda) + \frac 12 \log 2,
\end{align*}
and the claim follows from Lemma~\ref{lem:KM}.

\section{On potential improvements}

It remains to investigate the necessity of the various conditions in
Theorem~\ref{thm:main}. In particular, Theorem~\ref{thm:main} requires
a lower bound on $V$, an upper bound on $V$, and an upper bound on $D^2 V$.
We show that the two bounds on $V$ are necessary in some sense.

First of all, there is (in the one-dimensional setting)
a $V$ with $V'' \ge -1$ such that $\sup V < \infty$ but $e^{-V} d\gamma$
is not a Lipschitz image of $\gamma$. Indeed, define
\[
    V(x) = c - \begin{cases}
        0 & x < 1 \\
        \frac{(x - 1)^2}{2} & x \ge 1.
    \end{cases},
\]
where $c$ is a normalizing constant.
Hence, $e^{-V} d\gamma$ has density (with respect to the Lebesgue measure) proportional to $e^{-x}$ for $x \ge 1$,
and so if $X$ has distribution $e^{-V} d\gamma$ then $\Pr(X \ge x) \asymp e^{-x}$ as $x \to \infty$.
It follows that $e^{-V} d\gamma$ is not a Lipschitz image of $\gamma$, because if $X$
is distributed according to an $L$-Lipschitz image of $\gamma$ then $\Pr(X \ge x) = O(e^{-x^2/(2L^2)})$ as $x \to \infty$.

For an example showing the importance of an upper bound on $V$, we will construct
a family $\{V_T: T > 0\}$ of $-128$-concave functions $V_T$ that have a uniform
lower bound, such that for every $L > 0$ there is some $V_T$ for which $e^{-V_T}d\gamma$
is not an $L$-Lipschitz image of $\gamma$.

\begin{lemma}
    For every $T$ there is a function $V_T \ge -\log 2$ which is $-128$-concave such that
    $e^{-V_T} d\gamma$ is a probability measure that is not an $L$-Lipschitz image of $\gamma$
    for any $L < \frac{16}{17} \exp(\frac{95}{512} T^2)$.
\end{lemma}

\begin{proof}
Define $V_T$ by
\[
    V_T(x) = \max\left\{0, \frac{T^2}{4} - 64 (x-T)^2 \right\} - c_T
\]
where $c_T > 0$ is a normalizing constant. Note that $V_T$ is non-constant on the
interval $[\frac{15}{16} T, \frac{17}{16} T]$, and is constant elsewhere.
In particular, $V_T$ is constant on $(-\infty, 0]$, implying that
\[
    \frac 12 \le \int e^{-V_T - c_T} \, d\gamma = e^{-c_T},
\]
and hence $c_T \le \log 2$ for every $T$, meaning that $\inf V_T \ge -\log 2$ for every $T$.

In order to refute the existence of an $L$-Lipschitz map pushing $\gamma$ onto
$e^{-V_T} d\gamma$, note that if $\mu = g dx$ is an $L$-Lipschitz image of $\gamma$
then for every interval $[a, \infty)$ with $\mu([a, \infty)) \le \frac 12$,
$g(a) \ge \frac{1}{\sqrt{2\pi} L} \mu([a, \infty))$; this is a straightforward consequence
of the Gaussian isoperimetric inequality.
Now consider the set $A_T = [T, \infty)$ and let $\mu_T = e^{-V_T} d\gamma = g_T dx$.
On the one hand, $\sup e^{-V_T} \le 2$ and so $\mu_T(A_T) \le 2\gamma(A_T) \le \frac 12$ for sufficiently
large $T$. On the other hand, $\mu_T(A_T) \ge \mu_T([\frac{17}{16} T, \infty)) \ge \gamma([\frac{17}{16} T, \infty))
\ge \frac{16}{17 T} \exp(-\frac 12 (\frac{17}{16} T)^2)$, where the last inequality follows
by standard Gaussian tail bounds.
But $V_T(T) = \frac{T^2}{4} - c_T$, meaning that (recall that $g_T$ is the Lebesgue density of $\mu_T$)
\[
    g_T(T) = \frac{1}{\sqrt{2\pi}} e^{-\frac{3}{4} T^2 + c_T} \le \frac{1}{\sqrt{2\pi}} e^{-\frac{3}{4} T^2}
\]
If $\mu_T$ were an $L$-Lipschitz image of $\gamma$, the inequality $g(T) \ge \frac{1}{\sqrt{2\pi} L} \mu_T(A_T)$
would imply
\[
    e^{-\frac{3}{4} T^2} \ge \frac{16}{17 L} \exp\left(-\frac{17^2}{512} T^2\right).
\]
\end{proof}

Finally, concerning the assumption on $D^2 V$ in Theorem~\ref{thm:main},
Bobkov~\cite{Bobkov:10} showed that it is unnecessary in dimension 1: every
multiplicatively bound\-ed perturbation of a Gaussian measure is the Lipschitz
image of one. In the other direction, Colombo et al~\cite{CoFiJh:17} showed
that (in dimensions $n \ge 2$) if $V$ is bounded but $D^2 V$ is not, then the
Lipschitz constant of the Brenier map pushing $\gamma$ onto $e^{-V} d\gamma$
cannot be controlled. It remains open, however, whether there could be a
different $L$-Lipschitz map pushing $\gamma$ onto $e^{-V} d\gamma$ with $L$
depending only on $\|V\|_\infty$.
Certainly, this cannot be ruled out by appealing (like we did above)
to isoperimetric-type properties that are preserved by Lipschitz push-forwards,
since Bobkov showed that these properties do indeed hold whenever $V$ is bounded.

\bibliographystyle{plain}
\bibliography{lipschitz}

\end{document}

%% file: preamble.tex
\usepackage{amsmath}
\usepackage{amsfonts}
\usepackage{amsthm}
\usepackage{amssymb}
\usepackage{MnSymbol}
\usepackage{bbm}
\usepackage{listings}
\usepackage{enumerate}
\usepackage{cite}
\usepackage{etoolbox}
\usepackage{ifthen}

\newcommand{\R}{\mathbb{R}}

\newcommand{\E}{\mathbb{E}}

\newcommand{\inr}[2]{\langle #1, #2 \rangle}

\newcommand{\pdiffII}[3]{\ifstrequal{#2}{#3}
{\frac{\partial^2 #1}{\partial #2^2}}
{\frac{\partial^2 #1}{\partial #2 \partial #3}}
}
\newcommand{\diffII}[3]{\ifthenelse{\equal{#2}{#3}}
{\frac{d^2 #1}{d #2^2}}
{\frac{d^2 #1}{d #2 d #3}}
}

\newcommand{\diff}[2]{\frac{d #1}{d #2}}

\let\Pr\relax
\DeclareMathOperator{\Pr}{Pr}

\DeclareMathOperator{\Id}{Id}

\DeclareMathOperator{\supp}{supp}

\newcommand{\toD}{\stackrel{d}{\to}}

\newtheorem{theorem}{Theorem}[section]

\newtheorem{lemma}[theorem]{Lemma}

\newtheorem{proposition}[theorem]{Proposition}

\newtheorem{definition}[theorem]{Definition}

\newtheoremstyle{example}{\topsep}{\topsep}%
     {}%         Body font
     {}%         Indent amount (empty = no indent, \parindent = para indent)
     {\bfseries}% Thm head font
     {}%        Punctuation after thm head
     {\newline}%     Space after thm head (\newline = linebreak)
     {\thmname{#1}\thmnumber{ #2}\thmnote{ #3}}%         Thm head spec

\theoremstyle{example}